\documentclass{amsart}

\usepackage{amssymb}
\usepackage{amsmath,amsthm}
\usepackage{hyperref}
\usepackage[capitalize]{cleveref}

\DeclareMathOperator{\Sym}{Sym}

\DeclareMathOperator{\Cay}{Cay}
\DeclareMathOperator{\Aut}{Aut}
\DeclareMathOperator{\Z}{Z}
\DeclareMathOperator{\N}{N}

\newtheorem{theorem}{Theorem}[section]

\newtheorem{lemma}[theorem]{Lemma}
\newtheorem{cor}[theorem]{Corollary}

\theoremstyle{definition}
\newtheorem{definition}[theorem]{Definition}
\newtheorem{notation}[theorem]{Notation}
\newtheorem*{ack}{Acknowledgments}

\theoremstyle{remark}
\newtheorem{remark}[theorem]{Remark}

 \newcounter{case}
 \newenvironment{case}[1][\unskip]{\refstepcounter{case}\normalfont\em
 \medbreak \noindent Case \thecase\ #1.\ }{\unskip\upshape}
 \renewcommand{\thecase}{\arabic{case}}
 \crefformat{case}{Case~#2#1#3}

\newcommand{\conj}[2]{#2^{#1}}
\newcommand{\cartprod}{\mathbin{\scriptstyle\square}}
\newcommand{\ZZ}{\mathbb{Z}}

\newcommand{\pref}[1]{(\ref{#1})}

\begin{document}

\title[Detecting graphical regular representations]{Groups for which it is easy to detect graphical regular representations}

\author{Dave Witte Morris} 
\address{Department of Mathematics and Computer Science\\
University of Lethbridge\\
Lethbridge, AB. T1K 3M4}
\email{dave.morris@uleth.ca}

\author{Joy Morris} 
\address{Department of Mathematics and Computer Science\\
University of Lethbridge\\
Lethbridge, AB. T1K 3M4}
\email{joy.morris@uleth.ca}

\author{Gabriel Verret}
\address{Department of Mathematics, The University of Auckland, Private Bag 92019, 
Auckland 1142, New Zealand}
\email{g.verret@auckland.ac.nz}



\begin{abstract}
We say that a finite group~$G$ is \emph{DRR-detecting} if, for every subset~$S$ of~$G$, either the Cayley digraph $\Cay(G,S)$ is a digraphical regular representation (that is, its automorphism group acts regularly on its vertex set) or there is a nontrivial group automorphism~$\varphi$ of $G$ such that $\varphi(S) = S$. We show that every nilpotent DRR-detecting group is a $p$-group, but that the wreath product $\ZZ_p \wr \ZZ_p$ is not DRR-detecting, for every odd prime~$p$. We also show that if $G_1$ and $G_2$ are nontrivial groups that admit a digraphical regular representation and either $\gcd\bigl( |G_1|, |G_2| \bigr) = 1$, or $G_2$ is not DRR-detecting, then the direct product $G_1 \times G_2$ is not DRR-detecting.
Some of these results also have analogues for graphical regular representations.
\end{abstract}

\maketitle

\section{Introduction}

All groups and graphs in this paper are finite. Recall \cite{Babai} that a digraph~$\Gamma$ is said to be a \emph{digraphical regular representation (DRR)} of a group~$G$ if the automorphism group of~$\Gamma$ is isomorphic to~$G$ and acts regularly on the vertex set of $\Gamma$. If a DRR of~$G$ happens to be a graph, then it is also called a \emph{graphical regular representation (GRR)} of~$G$. Other terminology and notation can be found in \cref{PrelimSect}.

It is well known that if $\Gamma$ is a GRR (or DRR) of~$G$, then $\Gamma$ must be a Cayley graph (or Cayley digraph, respectively), so there is a subset~$S$ of~$G$ such that $\Gamma \cong \Cay(G,S)$ (and $S$ is inverse-closed if $\Gamma$ is a graph). It is traditional \cite[p.~243]{Godsil} to let
	\[ \Aut(G,S) = \{\, \varphi \in \Aut(G) \mid \varphi(S) = S \,\} . \] 
Since $\Aut(G,S) \subseteq \Aut \bigl( \Cay(G,S) \bigr)$, it is obvious (and well known) that if $\Aut(G,S)$ is nontrivial, then $\Cay(G,S)$ is not a GRR (or DRR). In this paper, we discuss groups for which the converse holds:

\begin{definition}
We say that a group $G$ is \emph{GRR-detecting} if, for every inverse-closed subset $S$ of $G$, $\Aut(G,S) = \{1\}$ implies that $\Cay(G,S)$ is a GRR. Similarly, a group $G$ is \emph{DRR-detecting} if for every subset $S$ of $G$, $\Aut(G,S) = \{1\}$ implies that $\Cay(G,S)$ is a DRR.
\end{definition}

\begin{remark}
Every Cayley graph is a Cayley digraph, so every DRR-detecting group is GRR-detecting.
\end{remark}

\begin{definition}
We say that a Cayley (di)graph $\Gamma=\Cay(G,S)$ on a group $G$ \emph{witnesses that $G$ is not GRR-detecting} (respectively, not DRR-detecting) if $\Aut(G,S) = \{1\}$ but $\Gamma$ is not a GRR (respectively, not a DRR) for $G$.
\end{definition}

An important class of DRR-detecting groups was found by Godsil. His result actually deals with vertex-transitive digraphs, rather than only the more restrictive class of Cayley graphs, but here is a special case of his result in our terminology:

\begin{theorem}[Godsil, cf.\ {\cite[Corollary 3.9]{Godsil}}] \label{Godsil}
Let $G$ be a $p$-group and let $\ZZ_p$ be the cyclic group of order~$p$. If $G$ admits no homomorphism onto the wreath product $\ZZ_p \wr \ZZ_p$ then $G$ is DRR-detecting (and therefore also GRR-detecting). 
\end{theorem}

Since $\ZZ_p \wr \ZZ_p$ is nonabelian, the following statement is an immediate consequence:

\begin{cor} \label{abelian-p}
Every abelian $p$-group is DRR-detecting (and therefore also GRR-detecting). 
\end{cor}

The following result shows that the bound in Godsil's theorem is sharp, in the sense that $\ZZ_p \wr \ZZ_p$ cannot be replaced with a larger $p$-group (when $p$ is odd):

\begin{theorem} \label{WreathNotDetect}
If $p$ is an odd prime, then the wreath product $\ZZ_p \wr \ZZ_p$ is not GRR-detecting (and is therefore also not DRR-detecting).
\end{theorem}

\begin{remark} \label{Z2wrZ2NotDetect}
The conclusion of \cref{WreathNotDetect} is not true for $p = 2$, because $\ZZ_2 \wr \ZZ_2$ is GRR-detecting. This is a special case of the fact that if a group has no GRR, then it is GRR-detecting \cite[Theorem~1.4]{Godsil-NonSolvable}.
\end{remark}

The following two results provide additional examples, by showing that direct products often yield groups that are not DRR-detecting:

\begin{theorem} \label{G1xG2relprime}
If $G_1$ and $G_2$ are nontrivial groups that admit a DRR (a GRR, respectively) and $\gcd \bigl( |G_1|, |G_2| \bigr) = 1$, then $G_1 \times G_2$ is not DRR-detecting (not GRR-detecting, respectively).
\end{theorem} 

\begin{theorem} \label{G1xG2notDRR}
If $G_1$ admits a DRR (a GRR, respectively) and $G_2$ is not DRR-detecting (not GRR-detecting, respectively), then $G_1 \times G_2$ is not DRR-detecting (not GRR-detecting, respectively).
\end{theorem}

These two results are the main ingredients in the proof of the following theorem:

\begin{theorem} \label{NilpotentGRR}
Every nilpotent DRR-detecting group is a $p$-group.
\end{theorem}

\begin{remark}
The phrase ``DRR-detecting'' in \cref{NilpotentGRR} cannot be replaced with ``GRR-detecting.'' For example, it is well known that every abelian group is GRR-detecting (unless it is an elementary abelian $2$-group), because the nontrivial group automorphism $x \mapsto x^{-1}$ is an automorphism of $\Cay(G,S)$.
\end{remark}

Here is an outline of the paper. 
A few definitions and basic results are recalled in \cref{PrelimSect}. 
\Cref{WreathNotDetect} is proved in \cref{ZpwrZpSect}. 
A generalization of \cref{G1xG2relprime} is proved in \cref{WreathProdSect}, by using wreath products of digraphs.
In \Cref{CartesionSect}, we recall some fundamental facts about cartesian products of digraphs and use them to prove \cref{G1xG2notDRR}. 
\Cref{NilpotentGRR} is proved in \cref{NilpotentSect}.

\section{Preliminaries} \label{PrelimSect}

\begin{definition}
Recall that if $S$ is a subset of a group~$G$, then the \emph{Cayley digraph of~$G$ (with respect to the connection set~$S$)} is the digraph $\Cay(G,S)$ whose vertex set is~$G$, such that there is a directed edge from $g_1$ to $g_2$ if and only if $g_2 = s g_1$ for some $s \in S$. If $S$ is closed under inverses, then $Cay(G,S)$ is a graph, and is called a \emph{Cayley graph}.
\end{definition}

See \cref{WreathGroupsDefn} for a general definition of the wreath product of two groups. The following special case is less complicated:

\begin{definition}
Let $\ZZ_p \wr \ZZ_p = \ZZ_p \ltimes (\ZZ_p)^p$, where $\ZZ_p$ acts on $(\ZZ_p)^p$ by cyclically permuting the coordinates: for $(v_1, v_2, \ldots, v_p) \in (\ZZ_p)^p$ and $g \in \ZZ_p$, we have
	\[ (v_1, v_2, \ldots, v_p)^g = (v_{g+1}, v_{g+2}, \ldots, v_n, v_1, v_2, \ldots, v_g) .\]
\end{definition}

We will use the following well-known results.

\begin{theorem}[Babai {\cite[Theorem~2.1]{Babai}}] \label{noDRR}
If a finite group does not admit a DRR, then it is isomorphic to
	\[ \text{$Q_8$, $(\ZZ_2)^2$, $(\ZZ_2)^3$, $(\ZZ_2)^4$, or $(\ZZ_3)^2$,} \]
where $Q_8$ is the quaternion group of order~$8$, which means
	\[ Q_8 = \langle\, i,j,k \mid i^2 = j^2 = k^2 = -1, \ ij = k, \ (-1)^2 = 1 \,\rangle . \]
\end{theorem}

\begin{lemma} \label{Norm(regrep)}
Let $\widehat{G}$ be the right regular representation of~$G$. Then:
	\begin{enumerate}
	\item $\widehat{G}$ is contained in $\Aut \bigl( \Cay(G,S) \bigr)$ for every subset~$S$ of~$G$.
	\item The normalizer of $\widehat{G}$ in $\Aut \bigl( \Cay(G,S) \bigr)$ is $\Aut(G,S) \ltimes \widehat{G}$. 
	\end{enumerate}
\end{lemma}

The latter has the following simple consequence:

\begin{lemma} \label{WitnessIffSelfnorm}
If $\Gamma$ is a Cayley digraph on~$G$ (a Cayley graph on~$G$, respectively), then $\Gamma$ witnesses that $G$ is not DRR-detecting (not GRR-detecting, respectively) if and only if the regular representation of~$G$ is a proper self-normalizing subgroup of $\Aut(\Gamma)$.
\end{lemma}

\section{\texorpdfstring{$\ZZ_p \wr \ZZ_p$}{Zp wreath Zp} is not GRR-detecting} \label{ZpwrZpSect}

Let $p$ be an odd prime. In this \lcnamecref{ZpwrZpSect}, we show that $\ZZ_p \wr \ZZ_p$ is not GRR-detecting. (This proves \cref{WreathNotDetect}.) To do this, we will construct a Cayley graph~$\Gamma$ on $\ZZ_p \wr \ZZ_p$ such that $\Gamma$ is not a GRR, but the regular representation of $\ZZ_p \wr \ZZ_p$ is self-normalizing in $\Aut(\Gamma)$.
In order to construct this graph, we first construct a certain group~$G$ that properly contains $\ZZ_p \wr \ZZ_p$. We will then define $\Gamma$ in such a way that $G$ is contained in $\Aut(\Gamma)$.

Let $A \cong \ZZ_p$ be a cyclic group of order~$p$, and choose an irreducible representation of~$A$ on a vector space $Q \cong (\ZZ_2)^n$ over the finite field with $2$~elements, such that $n \geq 2$. Now construct the corresponding semidirect product $A \ltimes Q$, which is a nonabelian group of order $2^n p$. 

Choose a nontrivial 1-dimensional representation $\chi \colon Q \to \{\pm1\} \subseteq \ZZ_p^\times$, and induce it to a representation of $A \ltimes Q$ on a vector space~$V$ over~$\ZZ_p$ \cite[\S3.3, pp.~28--30]{Serre}. Since $Q$ has index~$p$ in $A \ltimes Q$, the vector space~$V$ has dimension~$p$, so $V \cong (\ZZ_p)^p$. 
Let 
	\[ G = (A \ltimes Q) \ltimes V . \]

Since the representation of $A \ltimes Q$ on~$V$ is induced from a one-dimensional representation of the normal subgroup~$Q$, the restriction to~$Q$ decomposes as a direct sum of one-dimensional representations: $V = V_1 \oplus \cdots \oplus V_p$, where each $V_i$ is a subgroup of order~$p$ that is normalized by~$Q$ (cf.\ \cite[Proposition 22, p.~58]{Serre}).
(More precisely, for each~$i\in\{1,\ldots,p\}$, there is some $a \in A$, such that the representation of $Q$ on~$V_i$ is given by~$\chi^a$, where $\chi^a(g) = \chi(\conj{a^{-1}}{g})$ for $g \in Q$.)
Note that, since $A$ normalizes~$Q$, it must (cyclically) permute the $Q$-irreducible summands $V_1,\ldots,V_p$, so the Sylow $p$-subgroup $A \ltimes V$ of~$G$ is isomorphic to $\ZZ_p \wr \ZZ_p$.

Fix $a \in A^\times$. Since $A$ normalizes~$Q$, we know that $Qa$ is fixed by the action of~$Q$ on $Q \backslash G$. Also fix some $v_1 \in V_1^\times$. Then, for each~$i\in\{1,\ldots,p\}$, let $v_i = \conj{a^{i-1}}{v_1} \in V_i^\times$, and define $z = v_1 + v_2 + \cdots + v_p$, so $z$ is a generator of the center $\Z(A \ltimes V)$. 

Now let
	\[ S = \bigl( \langle v_1, v_2 \rangle \smallsetminus \langle v_1 \rangle \bigr) \cup (a \, \conj{Q}{z})^{\pm1} \subseteq A \ltimes V \subseteq G,\]
and let 
	\[ \Gamma = \Cay( A \ltimes V, S) . \]
Since $Q$ normalizes $\langle v_1 \rangle$ and $\langle v_2 \rangle$, and fixes the coset $Qa$ in $Q \backslash G$, it is clear that $S Q = Q S$. Therefore, after identifying the vertex set $A \ltimes V$ of~$\Gamma$ with $Q \backslash Q A V = Q \backslash G$ in the natural way, we have $G \subseteq \Aut(\Gamma)$, via the natural action of~$G$ on $Q \backslash G$. 
(Note that the action of~$G$ on $Q \backslash G$ is faithful, because $Q$ does not contain any nontrivial, normal subgroup of~$G$. Otherwise, since the action of~$A$ on~$Q$ is irreducible, the entire subgroup~$Q$ would have to be normal, which would mean that $Q$ acts trivially on $Q \backslash G$. But this is false, because the representation of~$Q$ on~$V$ is nontrivial.)
So $\Gamma$ is not a GRR.

Therefore, in order to show that $\ZZ_p \wr \ZZ_p \cong A \ltimes V$ is not GRR-detecting, it will suffice to show that $\Aut(A \ltimes V, S)$ is trivial.
To this end, let $\varphi$ be an automorphism of $A \ltimes V$ that fixes~$S$. 
We will show that $\varphi$ is trivial.

Since $V$ is characteristic in $A \ltimes V$ (for example, it is the only abelian subgroup of order~$p^p$), we know that 
	\[ \varphi(V \cap S) = V \cap S = \langle v_1, v_2 \rangle \smallsetminus \langle v_1 \rangle \subseteq \langle v_1, v_2 \rangle .\]
So 
	\[ \varphi \bigl( \langle v_1, v_2 \rangle \bigr)
	= \varphi \bigl( \langle v_1 v_2, v_2 \rangle \bigr)
	= \langle \varphi (v_1 v_2) , \varphi(v_2) \rangle 
	\subseteq \langle \varphi(V \cap S) \rangle
	\subseteq \langle v_1, v_2 \rangle .\]
Since $\varphi$ is injective, we conclude that $\varphi$ fixes $\langle v_1, v_2 \rangle$ (setwise).
Then $\varphi$ also fixes $\langle v_1, v_2 \rangle \smallsetminus S = \langle v_1 \rangle$.

We have $\varphi(a) \notin V$ (because $a \notin V$ and $\varphi$ fixes~$V$), which means $\varphi(a) = a^k v' $ for some $k \in \ZZ_p^\times$ and $v' \in V$. Then (since $v'$ centralizes~$V$, because $V$ is abelian) we have
	\[ \langle v_1, v_2 \rangle
	= \varphi \bigl( \langle v_1, v_2 \rangle \bigr)
	\ni \varphi(v_2)
	= \varphi(\conj{a}{v_1})
	= \conj{\varphi(a)}{\varphi(v_1)}
	\in \conj{a^k}{\langle v_1 \rangle}
	= \langle v_{k + 1} \rangle
	,\] 
so $k \in \{0,1\} \cap \ZZ_p^\times = \{1\}$, which means 
	\[ \varphi(a) = a v' . \]
Note that (since $\varphi(V) = V$) this implies
	\[ \varphi(aV) = aV .\]

Since $\varphi$ fixes~$\langle v_1 \rangle$, we have $\varphi(v_1) = \ell v_1$ for some $\ell \in \ZZ_p^\times$. For every $i \in \{1,\ldots,p\}$, this implies
	\[ \varphi(v_i)
	= \varphi(\conj{a^{i-1}}{v_1})
	= \conj{\varphi(a^{i-1})}{\varphi(v_1)}
	= \conj{a^{i-1}}{(\ell v_1)}
	= \ell v_i
	. \]
Since $\{v_1,\ldots,v_p\}$ generates~$V$, we conclude that 
	\[ \text{$\varphi(v) = \ell v$ for all $v \in V$.} \]

To complete the proof, we will show that $v'$ is trivial and $\ell = 1$. (This means that $\varphi$ fixes~$a$, and also fixes every element of~$V$. So $\varphi$ is the trivial automorphism, as desired.)
For all $z_0 \in \conj{Q}{z}$, we have
	\begin{align*}
	 a \cdot (v' + \ell z_0)
	&= a \, v' \cdot (\ell z_0)
	= \varphi(a) \, \varphi(z_0)
	= \varphi(a \, z_0) 
	\\& \in \varphi(S \cap aV) 
	= \varphi(S) \cap \varphi(aV) 
	= S \cap a V
	= a \, \conj{Q}{z} 
	. \end{align*}
	
Therefore, if we write $v' = \sum_{i=1}^p s_i v_i$ (with $s_i \in \ZZ_p$)
and $z_0 = \sum_{i=1}^p t_i v_i$ (with $t_i \in \{\pm 1\}$), then we have 
	\[ \text{$s_i + \ell t_i \in \{\pm1\} \pmod{p}$ for every~$i$.} \]
For any given $i$, the representation of~$Q$ on~$V_i$ is nontrivial, so we may choose $z_0$ so that $t_i = -1$. Therefore, we have $s_i -\ell \equiv \pm 1 \pmod{p}$. On the other hand, by letting $z_0 = z$ (and noting that $s_i - \ell \not\equiv s_i + \ell \pmod{p}$) we see that we also have $s_i + \ell \equiv \mp 1 \pmod{p}$. Adding these two equations and dividing by~$2$ yields $s_i = 0$ (for all~$i$). So $v'$ is trivial (which means $\varphi(a) = a$). 

All that remains is to show that $\ell = 1$ (which means that $\varphi$ acts trivially on~$V$). 
Suppose this is not true. (That is, suppose $\ell \neq 1$.)
For convenience, let $Z = \langle z \rangle = \Z(A \ltimes V)$. Note that, since $\varphi(a) = a$, we have 
	\[ a \cdot (\ell z) = \varphi(az) \in \varphi( S \cap aV) = S \cap aV = a \, \conj{Q}{z}, \]
so there is some $g \in Q$, such that $\conj{g}{z} = \ell z$. Since $Z = \langle z \rangle$, this implies that $g$ is an element of the normaliser $\N_Q(Z)$ of $Z$ in $Q$. Also note that $g$ is nontrivial, because $\ell \neq 1$. Then, since $\N_Q(Z)$ is normalized by~$A$ (because $A$ normalizes $Q$ and~$Z$), the irreducibility of the representation of~$A$ on~$Q$ implies that $\N_Q(Z) = Q$. Hence, $Q$ acts on~$Z$ by conjugation, so $Q/C_Q(Z)$ embeds in the cyclic group $\Aut(Z) \cong \ZZ_p^\times$. Since $Q$ is an elementary abelian $2$-group, this implies that $|Q/C_Q(Z)| \le 2$. It is clear that $|Q| \geq 4$ (because $Q \cong (\ZZ_2)^n$ and $n \geq 2$), so we conclude that $C_Q(Z)$ is nontrivial. Using once again the fact that the representation of~$A$ on~$Q$ is irreducible, we conclude that $C_Q(Z) = Q$, which means that $Q$ centralizes~$Z$. However, since 	
	\[ Z 
	= \langle z \rangle
	= \langle v_1 + v_2 + \cdots + v_p \rangle
	,\]
and each $\langle v_i \rangle = V_i$ is a $Q$-invariant subspace, this implies that $Q$ centralizes each~$v_i$, and is therefore trivial on~$V$. On the other hand, we have $z^g= \ell z \neq z$ (since $\ell \neq 1$), and $g \in Q$. This is a contradiction.

\section{Using wreath products to construct witnesses} \label{WreathProdSect}

In this \lcnamecref{WreathProdSect}, we prove \cref{DRR-normal}, which is a generalization of \Cref{G1xG2relprime}.

\begin{notation}
In this \lcnamecref{WreathProdSect}, $N$ always denotes a normal subgroup of a group~$G$, and $\overline{\phantom{x}} \colon G \to G/N$ the natural homomorphism.
\end{notation}

\begin{notation}
For each $c \in G$ and $f \colon \overline{G} \to N$, we let $\varphi_{c,f}$ be the permutation on~$G$ that is defined by
	\[ \text{$\varphi_{c,f}(x) = xc \, f(\overline{x})$ for $x \in G$.} \]
Let $W(G,N)$ be the set of all such permutations of~$G$.
\end{notation}

\begin{remark}
Informally speaking, an element of $W(G,N)$ is defined by choosing an element of~$\overline{G}$ (or, more accurately, by choosing a coset representative) to permute the cosets of~$N$, and then choosing an element of~$N$ to act on each coset. (The elements of~$N$ can be chosen independently on each coset.)

We have $\varphi_{c,f} = \varphi_{c',f'}$ if and only if there is some $n \in N$, such that $c' = cn$ and $f'(\overline{x}) = n^{-1} f(x)$ for all~$\overline{x}$. From this, it follows that $|W(G,N)| = |\overline{G}| \cdot |N|^{|\overline{G}|}$.
\end{remark}

\begin{remark} \label{WreathGroupsDefn}
The usual definition of the \emph{wreath product} of two groups $K$ and~$H$ is essentially:
	\[ K \wr H = W( K \times H, \{1\} \times H) . \]
\end{remark}

\begin{definition}
Recall that the \emph{wreath product} $X \wr Y$ of two (di)graphs $X$ and~$Y$ is the (di)graph whose vertex set is the cartesian product $X \times Y$, and with a (directed) edge from $(x_1,y_1)$ to $(x_2,y_2)$ if and only if either there is a (directed) edge from $x_1$ to~$x_2$ or $x_1 = x_2$ and there is a (directed) edge from $y_1$ to~$y_2$.
\end{definition}

The following two observations are well known (and fairly immediate from the definitions).
The first is a concrete version of the Universal Embedding Theorem, which states that $G$ is isomorphic to a subgroup of $(G/N) \wr N$.

\begin{lemma} \label{WreathGroups}
$W(G,N)$ is a subgroup of the symmetric group on~$G$. It is isomorphic to the wreath product $\overline{G} \wr N$, and contains the regular representation of~$G$.
\end{lemma}

\begin{lemma} \label{WreathCayley}
Suppose $\Cay(\overline{G}, \overline{S_1})$ is a loopless Cayley digraph on~$\overline{G}$, and $\Cay(N, S_2)$ is a Cayley digraph on~$N$. Let $S_1 = \{\, g \in G \mid \overline{g} \in \overline{S_1} \,\}$. Then
	\[ \Cay( G, S_1 \cup S_2) \cong \Cay(\overline{G}, \overline{S_1}) \wr \Cay(N, S_2) , \]
and $W(G,N)$ is contained in the automorphism group of $\Cay( G, S_1 \cup S_2)$.
\end{lemma}

The following result is a special case of the general principle that the automorphism group of a wreath product of digraphs is usually the wreath product of the automorphism groups. We have stated it only for DRRs, making use of some straightforward observations about the automorphism group of a DRR on more than $2$ vertices, but the much more general statement in \cite{DobsonM} applies to all vertex-transitive digraphs.

\begin{lemma}[cf.\ Dobson-Morris {\cite[Theorem 5.7]{DobsonM}}] \label{thm:wreath-aut}
Assume that $\Cay(\overline{G}, \overline{S_1})$ and $\Cay(N, S_2)$ are loopless DRRs, and let $S_1$ be as in \cref{WreathCayley}. If either $|\overline{G}| \neq 2$ or $|N| \neq 2$, then 
	\[ \Aut \bigl( \Cay( G, S_1 \cup S_2) \bigr) = W(G,N) . \]
\end{lemma}

In light of \cref{WitnessIffSelfnorm,thm:wreath-aut}, it is of obvious interest to us to determine when the regular representation of~$G$ is self-normalizing in $W(G,N)$. Our next result is the answer to this question.

\begin{theorem} \label{SelfNormalizingInWreath}
Let $N$ be a normal subgroup of~$G$. Then the regular representation of~$G$ is self-normalizing in $W(G,N)$ if and only if
	\begin{enumerate}
	\item \label{SelfNormalizingInWreath-contained}
	$\Z(N) \le \Z(G)$, 
	and
	\item \label{SelfNormalizingInWreath-prime}
	the order of the abelianization of $G/N$ is relatively prime to $|{\Z(N)}|$.
	\end{enumerate}
\end{theorem}

\begin{proof}
($\Rightarrow$)
We prove the contrapositive.
	\pref{SelfNormalizingInWreath-contained}~If $\Z(N) \not\le \Z(G)$, then there exists $n \in \Z(N)$ such that $n \notin \Z(G)$. Conjugation by~$n$ is an element of $W(G,N)$ that normalizes the right regular representation of~$G$, but is not in the right regular representation of~$G$. 
	\pref{SelfNormalizingInWreath-prime}~If the order of the abelianization of $\overline{G}/N$ is not relatively prime to $|\Z(N)|$, then there is a nontrivial homomorphism $f \colon \overline{G} \to \Z(N)$. We may assume that hypothesis~\pref{SelfNormalizingInWreath-contained} is satisfied, and then it is straightforward to verify that the corresponding element $\varphi_{f,1}$ of $W(G,N)$ normalizes the right regular representation of~$G$:
	\begin{align*}
	 \varphi_{f,1}(x g) 
	&= xg \, f( \overline{xg})
		&& \text{(definition of $\varphi_{f,1}$)}
	\\&= x \, f( \overline{xg}) \, g
		&& \text{($f(\overline{xg}) \in f(\overline{G}) \subseteq \Z(N) \subseteq \Z(G)$)}
	\\&= x \, f( \overline{x} ) \, f( \overline{g}) \, g
		&& \text{($f$ is a homomorphism)}
	\\&= 	\varphi_{f,1}(x) \cdot f( \overline{g}) \, g
		&& \text{(definition of $\varphi_{f,1}$)}
	. \end{align*}

($\Leftarrow$) By \cref{Norm(regrep)}, it suffices to show that $\Aut(G) \cap W(G,N)$ is trivial.
To this end, let $\varphi \in \Aut(G) \cap W(G,N)$.
Since $\varphi \in W(G,N)$, there exist $c \in G$ and $f \colon \overline G \to N$, such that
	\[ \text{$\varphi(x) = x c \, f(\overline{x})$ for all $x \in G$.} \]

Since $\varphi$ is a group automorphism we know $\varphi(1) = 1 \in N$, so we may assume $c = 1$, after multiplying~$c$ on the right by an element of~$N$. Then we must have $f(\overline{1}) = 1$. Now, for each $n \in N$, we have $\overline{n} = \overline{1}$, so
	\[ \varphi(n) = n \cdot f(\overline{n}) = n \cdot f(\overline{1}) = n \cdot 1 = n .\]

Therefore, for all $g \in G$ and $n \in N$, we have
	\begin{align*}
	gn \cdot f(\overline{g})
	&= gn \cdot f(\overline{gn})
	= \varphi(gn)
	= \varphi(g) \, \varphi(n)
	= g \, f(\overline{g}) \cdot n
	, \end{align*}
so $n \cdot f(\overline{g}) = f(\overline{g}) \cdot n$. 
Since this is true for all $n \in N$, we conclude that $f(\overline{g}) \in \Z(N)$. 
Since $\Z(N) \subseteq \Z(G)$, this implies $f(\overline{g}) \in \Z(G)$ for all $\overline{g}$. Therefore, for all $g,h \in G$, we have
	\begin{align*}
	gh \cdot f(\overline{gh})
	&= \varphi(gh)
	= \varphi(g) \, \varphi(h)
	= g \, f(\overline{g}) \cdot h \, f(\overline{h})
	= g h \cdot f(\overline{g}) \, f(\overline{h})
	. \end{align*}
So $f$ is a group homomorphism. Since $f(\overline{G} )$ is contained in $\Z(N)$, which is abelian, we see from~\pref{SelfNormalizingInWreath-prime} that $f$ must be trivial. Since $c$ is also trivial, we conclude that $\varphi(x) = x$ for all~$x$. Since $\varphi$ is an arbitrary element of $\Aut(G) \cap W(G,N)$, this completes the proof.
\end{proof}

\begin{remark}
A slight modification of the proof of \cref{SelfNormalizingInWreath} shows that if $\widehat G$ is the right regular representation of~$G$, then the normalizer of $\widehat G$ in $W(G,N)$ is 
	\[ \bigl\{\, \varphi_{c,f} \mid c \in G, f \in Z^1\bigl( \overline{G}, \Z(N) \bigr) \,\bigr\} ,\]
where 
	\[ Z^1\bigl( \overline{G}, \Z(N) \bigr)
	= \{\, f \colon \overline{G} \to \Z(N) \mid 
	\text{$f(\overline{gh}) = f(\overline{g})^{\overline{h}} \, f(\overline{h})$ 
	for all $\overline{g}, \overline{h} \in \overline{G}$} \,\} \]
is the set of all ``$1$-cocycles'' or ``crossed homomorphisms'' from $\overline{G}$ to $\Z(N)$ (in the terminology of group cohomology \cite{Wikipedia-GrpCoho}). This fact is presumably known.
\end{remark}

It may also be of interest to note that hypotheses \pref{SelfNormalizingInWreath-contained} and~\pref{SelfNormalizingInWreath-prime} in \cref{SelfNormalizingInWreath} are obviously satisfied when $\Z(N)$ is trivial.

Combining the results of this section, we obtain the following.

\begin{cor}\label{DRR-normal}
Let $N$ be a nontrivial, proper, normal subgroup of~$G$, such that $N$ and $G/N$ each admit a DRR (or, respectively, a GRR). If 
	\begin{enumerate}
	\item \label{DRR-normal-contained}
	$\Z(N) \le \Z(G)$, 
	and
	\item \label{DRR-normal-prime}
	the order of the abelianization of $G/N$ is relatively prime to $|{\Z(N)}|$,
	\end{enumerate}
then $G$ is not DRR-detecting (respectively, not GRR-detecting).

More precisely, if we let $\Gamma_1$ be a DRR (respectively, GRR) on $G/N$ and $\Gamma_2$ be a DRR (respectively, GRR) on~$N$, then $\Gamma_1 \wr \Gamma_2$ witnesses that $G$ is not DRR-detecting (respectively, not GRR-detecting).
\end{cor}
\begin{proof}
Clearly, either $|\overline{G}| \neq 2$ or $|N| \neq 2$. It then follows by \cref{WreathCayley} and \cref{thm:wreath-aut} that $\Aut \bigl( \Gamma_1 \wr \Gamma_2 ) = W(G,N)$. By \cref{SelfNormalizingInWreath}, the regular representation of~$G$ is self-normalizing in $W(G,N)$, therefore $\Gamma_1 \wr \Gamma_2$ witnesses that $G$ is not DRR-detecting (respectively, not GRR-detecting).
\end{proof}

 Note that \cref{G1xG2relprime} can be obtained from \cref{DRR-normal} by letting $G=G_1\times G_2$ and $N = G_2$.

\section{Using cartesian products to construct witnesses} \label{CartesionSect}

\begin{definition}
Recall that the \emph{cartesian product} $X \cartprod Y$ of two (di)graphs $X$ and~$Y$ is the (di)graph whose vertex set is the cartesian product $X \times Y$, such that there is a (directed) edge from $(x_1,y_1)$ to $(x_2,y_2)$ if and only if either 
$x_1 = x_2$ and there is a (directed) edge from $y_1$ to~$y_2$, or
$y_1 = y_2$, and there is a (directed) edge from $x_1$ to~$x_2$.
\end{definition}

We say that a (di)graph is \emph{prime} (with respect to cartesian product) if it has more than one vertex, and is not isomorphic to the cartesian product of two (di)graphs, each with more than one vertex. 
It is well known that every (di)graph can be written uniquely as a cartesian product of prime factors (up to a permutation of the factors), but we do not need this fact. 

To avoid the need to consider permutations of the factors, the following result includes the hypothesis that the factors are pairwise non-isomorphic. (This is not assumed in~\cite{Walker}, which also considers isomorphisms between two different cartesian products, instead of only automorphisms of a single digraph.) The upshot is that, in this situation, the automorphism group of the cartesian product is the direct product of the automorphism groups.

\begin{theorem}[Walker, cf.\ {\cite[Theorem~10]{Walker}}]\label{strict-factorisation}
Let $\Gamma_1,\ldots,\Gamma_k$ be weakly connected prime digraphs that are pairwise non-isomorphic. If $\varphi$ is an automorphism of $\Gamma_1 \cartprod \cdots \cartprod \Gamma_k$, then for each~$i$, there is an automorphism~$\varphi_i$ of~$\Gamma_i$ such that, for every vertex $(v_1,\ldots,v_k)$ of $\Gamma_1 \cartprod \cdots \cartprod \Gamma_k$, we have
	\[ \varphi(v_1,\ldots,v_k) = \bigl( \varphi_1(v_1), \ldots, \varphi_k(v_k) \bigr) . \]
\end{theorem}

Prime graphs are quite abundant:

\begin{theorem}[Imrich {\cite[Theorem 1]{Imrich}}]\label{cartesian-GRRs}
If $\Gamma$ is a graph (with more than one vertex), such that neither $\Gamma$ nor its complement~$\overline{\Gamma}$ is prime, then $\Gamma$ is one of the following:
\begin{enumerate}
\item \label{cartesian-GRRs-C4}
the cycle of length $4$ or its complement (two disjoint copies of $K_2$);
\item \label{cartesian-GRRs-cube}
the cube or its complement (the graph $K_2 \times K_4$);
\item \label{cartesian-GRRs-K3xK3}
$K_3 \cartprod K_3$ (which is self-complementary); or
\item \label{cartesian-GRRs-K2xDelta}
$K_2 \cartprod \Delta$, where $\Delta$ is the graph obtained by deleting an edge from $K_4$ (which is self-complementary).
\end{enumerate}
\end{theorem}

The following is an analogous result for digraphs. (Recall that a digraph is \emph{proper} if it is not a graph.)

\begin{theorem}[Grech-Imrich-Krystek-Wojakowski {\cite[Theorem 1.2]{GrechEtAl}} and 
Morgan-{\allowbreak}Morris-Verret {\cite[Theorem 2.2]{MMV-explosive}}]\label{primality}
If $\Gamma$ is a proper digraph, then at least one of $\Gamma$ or $\overline{\Gamma}$ is prime.
\end{theorem}

\begin{cor}\label{graph-prime}
If a nontrivial group $G$ admits a DRR (respectively, GRR), then it admits a DRR (respectively, GRR) that is prime (and weakly connected). Furthermore, if $G$ is not DRR-detecting (respectively, not GRR-detecting), then there is a witness that is prime (and weakly connected).
\end{cor}

\begin{proof}
First, note that $\Gamma$ and $\overline{\Gamma}$ have the same automorphism group, so $\Gamma$ is a DRR (GRR, respectively) for $G$ if and only if $\overline{\Gamma}$ is. Similarly, $\Gamma$ is a witness that $G$ is not DRR-detecting (GRR-detecting, respectively) if and only if $\overline{\Gamma}$ is. 

Also note that if a prime digraph~$\Gamma$ is not weakly connected, and is either a DRR or a witness that some group is not DRR-detecting, then $\overline{\Gamma} = K_2$ (so $\overline{\Gamma}$ is prime and weakly connected). 
This is because any vertex-transitive digraph is isomorphic to $\Gamma_0 \cartprod \overline{K_n}$, where $\Gamma_0$ is a weakly connected component of the digraph, and $n$~is the number of components. 

Suppose that $\Gamma$ is a GRR for $G$. By \cref{cartesian-GRRs}, at least one of $\Gamma$ or $\overline{\Gamma}$ is prime with respect to cartesian product, unless $\Gamma$ is one of the graphs listed in that \lcnamecref{cartesian-GRRs}, but none of those graphs is a GRR, because the automorphism group does not act regularly on the set of vertices:
\begin{enumerate}
\item the automorphism group of a cycle of length $4$ (or its complement) is the dihedral group of order $8$;
\item the automorphism group of the cube (or its complement) is $\ZZ_2 \times \Sym(4)$, of order $48$;
\item the automorphism group of $K_3 \cartprod K_3$ is $\ZZ_2 \wr \Sym(3)$, of order~$72$; and \item the graph $K_2 \cartprod \Delta$ is not vertex-transitive (it is not even true that all vertices have the same valency).
\end{enumerate}

Now, suppose that $\Gamma$ is a DRR for~$G$. We may assume that $\Gamma$ is a proper digraph. (Otherwise, $\Gamma$ is a GRR, so the preceding paragraph applies.) Then, by \cref{primality}, either $\Gamma$ or $\overline{\Gamma}$ is prime with respect to cartesian product. 

Finally, suppose $\Gamma$ is a witness that $G$ is not DRR-detecting (or not GRR-detecting, respectively), such that neither $\Gamma$ nor~$\overline{\Gamma}$ is prime. 
This implies that $\Gamma$ is one of the graphs listed in \cref{cartesian-GRRs}.
(So $G$ is not GRR-detecting.)

However, it is easy to see that none of the graphs listed in \cref{cartesian-GRRs} is a witness. First, recall that a $p$-subgroup of a group cannot be self-normalizing unless it is a Sylow subgroup. Therefore (by \cref{WitnessIffSelfnorm}), if a graph~$\Gamma$ of prime-power order $p^k$ is a witness that some group is not GRR-detecting, then $p^k$ must be the largest power of~$p$ that divides $\Aut(\Gamma)$. This shows that the graphs in \pref{cartesian-GRRs-C4} and~\pref{cartesian-GRRs-cube} are not witnesses. If $\Gamma$ is as described in~\pref{cartesian-GRRs-K3xK3}, then the only regular subgroup of $\Aut(\Gamma)$ is the unique (Sylow) subgroup of order~$9$, which is normal, and is therefore obviously not self-normalizing. Finally, as noted above, the graphs in~\pref{cartesian-GRRs-K2xDelta} are not vertex-transitive.
\end{proof}

\begin{proof}[\bf Proof of \cref{G1xG2notDRR}]
For simplicity, we consider only DRRs (because the proof is the same for GRRs).
Let $\Gamma_1 = \Cay(G_1,S_1)$ be a DRR for~$G_1$, and let $\Gamma_2 = \Cay(G_2,S_2)$ be a witness that $G_2$ is not DRR-detecting. By \cref{graph-prime}, we may assume that $\Gamma_1$ and~$\Gamma_2$ are prime with respect to cartesian product (and are weakly connected). Since $\Gamma_1$ is a DRR, but $\Gamma_2$ is not, we know that $\Gamma_1 \not\cong \Gamma_2$. Therefore, we see from \cref{strict-factorisation} that $\Aut(\Gamma_1 \cartprod \Gamma_2) = \Aut(\Gamma_1) \times \Aut(\Gamma_2)$. 

Since $\Gamma_2$ is not a DRR, $\Gamma_1 \cartprod \Gamma_2$ is not a DRR. Similarly, since the regular representation of~$G_1$ is all of $\Aut(\Gamma_1)$ and the regular representation of~$G_2$ is self-normalizing in $\Aut(\Gamma_2)$, the regular representation of $G_1 \times G_2$ is self-normalizing in $\Aut(\Gamma_1 \cartprod \Gamma_2)$. So $\Gamma_1 \cartprod \Gamma_2$ is a witness that $G_1 \times G_2$ is not DRR-detecting.
\end{proof}

\section{Nilpotent DRR-detecting groups are \texorpdfstring{$p$}{p}-groups} \label{NilpotentSect}

In this \lcnamecref{NilpotentSect}, we prove \cref{NilpotentGRR}, which states that if a nilpotent group is not a $p$-group, then it is not DRR-detecting. 
In most cases, this follows easily from \cref{G1xG2relprime,G1xG2notDRR}, but there is one special case that requires a different proof:

\begin{lemma}\label{Q8}
If $H$ is a nontrivial group of odd order and $H \not\cong \ZZ_3 \times \ZZ_3$, then $Q_8 \times H$ is not DRR-detecting.
\end{lemma}

\begin{proof}
From \cref{noDRR}, we see that $H$ admits a DRR (because it has odd order, but is not $\ZZ_3 \times \ZZ_3$), so we may let $\Cay(H, S_1)$ be a DRR. Let $S=S_1 \cup \{i\}\cup jH\subseteq G$. It suffices to show that $\Aut(G,S) = \{1\}$ and that $\Cay(G,S)$ is not a DRR.

Let $\varphi \in \Aut(G,S)$. We can characterise $S_1$ as the set of all elements of $S$ that have odd order. Thus, we must have $\varphi(S_1)=S_1$, so $H = \langle S_1 \rangle$ is fixed setwise by~$\varphi$. Since the identity vertex is also fixed and the induced subgraph on $H$ is a DRR, every element of $H$ must be fixed by $\varphi$. We can use this fact to distinguish $i$ from the elements of $jH$ (all of which differ from each other by elements of $H$), so $i$ is fixed by $\varphi$. Finally, $j$ is the unique element of order $4$ in $jH$, so it too is fixed by $\varphi$. We now know that $\varphi$ is an automorphism of $G$ that fixes every element of a generating set for $G$. So $\varphi$ must be trivial.

All that remains is to show that $\Cay(G,S)$ is not a DRR. Fix a nontrivial element $h \in H$, and define a permutation~$\tau$ of~$G$ by
	\[ \tau(x) = \begin{cases}
		\hfil x & \text{if $x \in \langle H,i\rangle$}; \\
		xh & \text{if $x \in j\langle H,i \rangle$}
		. \end{cases} \]
Note that $\tau$ is a permutation of~$G$, because right multiplication by~$h$ is a permutation of $G$ that fixes $\langle H,i\rangle$ setwise.

We claim that $\tau$ is an automorphism of $\Cay(G,S)$. First, note that a directed edge of the form $g \to s_1 g$ or $g \to ig$ either has both of its endpoints in $\langle H,i\rangle$, or has both of its endpoints in $j \langle H,i\rangle$. Since right multiplication by~$h$ is an automorphism of $\Cay(G,S)$, it is clear that $\tau$ preserves such directed edges. The remaining directed edges are of the form $g \to gjh'$ for some $h'\in H$. Multiplying either $g$ or~$gjh'$ on the right by~$h$ results in another such directed edge. This completes the proof that $\tau$ is an automorphism of $\Cay(G,S)$.
\end{proof}

\begin{proof}[\bf Proof of \cref{NilpotentGRR}]
Let $G$ be a nilpotent group, and assume that $G$ is not a $p$-group. 
(Note that $|G|$ is divisible by at least two distinct primes.)
We will show that $G$ is not DRR-detecting.

\setcounter{case}{0}

\begin{case} \label{NilpotentGRRPf-primes}
$|G|$ is divisible by at least three distinct primes.
\end{case}
Let $p$ be the largest prime divisor of~$|G|$ and let $P$ be a Sylow $p$-subgroup of~$G$. Since $G$ is nilpotent, we may write $G = P \times H$ for some subgroup~$H$ with $\gcd \bigl( |P|, |H| \bigr) = 1$. Since $p$ is the largest of at least three primes dividing $|G|$, neither $P$ nor $H$ is a $2$-group or a $3$-group, so we see from \cref{noDRR} that $P$ and $H$ each admit a DRR. Therefore, \cref{G1xG2relprime} implies that $G = P \times H$ is not DRR-detecting.

\begin{case}
 $|G|$ is divisible by precisely two distinct primes $p$ and~$q$.
\end{case}
Since $G$ is nilpotent, we have $G = P \times Q$, where $P$ is a Sylow $p$-subgroup and $Q$ is a Sylow $q$-subgroup of $G$. If $P$ and~$Q$ each admit a DRR, then \cref{G1xG2relprime} implies that $G = P \times Q$ is not DRR-detecting.

We may thus assume, without loss of generality, that $P$ does not admit a DRR. Using \cref{noDRR} and \cref{Q8} and interchanging $P$ and $Q$ if necessary, we may assume that $P$ is isomorphic to one of $(\ZZ_2)^2$, $(\ZZ_2)^3$, $(\ZZ_2)^4$, or $(\ZZ_3)^2$. Thus, we may write $P = (\ZZ_p)^r$, with $r \geq 2$. 

Since $(\ZZ_p)^{r-1} \times Q$ is not a $p$-group, we may assume, by induction on $|G|$, that it is not DRR-detecting. Also note that $\ZZ_p$ admits a DRR. (Take the directed $p$-cycle $\overrightarrow{C_p}$ if $p \geq 3$; or take $K_2$ if $p = 2$.) Therefore, by applying \cref{G1xG2notDRR} with $G_1 = \ZZ_p$ and $G_2 = (\ZZ_p)^{r-1} \times Q$, we see that the group $G = \ZZ_p \times \bigl( (\ZZ_p)^{r-1} \times Q \bigr)$ is not DRR-detecting. 
\end{proof}

\begin{ack}
This work was supported in part by the Natural Science and Engineering Research Council of Canada (grant RGPIN-2017-04905), and Gabriel Verret is grateful to the N.Z. Marsden Fund for its support (via grant UOA1824).
\end{ack}


\begin{thebibliography}{99}

\bibitem{Babai} L.~Babai. Finite digraphs with given regular automorphism groups. \textit{Periodica Mathematica Hungarica} \textbf{11} (1980), 257--270.

\bibitem{DixonMortimer} J.~Dixon, B.~Mortimer. Permutation Groups. Graduate texts in Mathematics, 163. \textit{Springer-Verlag, New York}. 1996.

\bibitem{DobsonM} E.~Dobson, J.~Morris. Automorphism groups of wreath product digraphs. \textit{Electronic. J. Combin.} \textbf{16} (2009), \#R17.

\bibitem{Godsil-NonSolvable} C.~Godsil. 
GRRs for nonsolvable groups. 
\emph{Algebraic Methods in Graph Theory, Vol.~I,~II (Szeged, 1978)}, pp. 221--239.
North-Holland, Amsterdam, 1981.

\bibitem{Godsil} C.~Godsil. On the full automorphism group of a graph. \textit{Combinatorica} \textbf{1} (1981), 243--256.

\bibitem{GrechEtAl} M.~Grech, W.~Imrich, A.~D.~Krystek, L.~J.~Wojakowski. Direct product of automorphism groups of digraphs. \textit{Ars Math. Contemp.} \textbf{17} (2019), 89--101.

\bibitem{Imrich} W.~Imrich. On products of graphs and regular groups. \textit{Israel J. Math} \textbf{11} (1972), 258--264.

\bibitem{MMV-explosive} L.~Morgan, J.~Morris, G.~Verret. Digraphs with small automorphism groups that are Cayley on two nonisomorphic groups. \textit{Art of Discrete and Applied Math.}, to appear.

\bibitem{MorrisSpiga} J.~Morris, P.~Spiga. Classification of finite groups that admit an oriented regular representation. \textit{Bull. London Math. Soc.} \textbf{50} (2018), 811--831.

\bibitem{Serre} J.--P.~Serre. Linear Representations of Finite Groups. Graduate Texts in Mathematics, 42. \textit{Springer-Verlag, New York}. 1977

\bibitem{Walker} J.~W.~Walker. Strict refinement for graphs and digraphs. \textit{J.~Combin. Theory Ser. B} \textbf{43} (1987), 140--150.

\bibitem{Wikipedia-GrpCoho}
Wikipedia. Group cohomology.
\url{https://en.wikipedia.org/wiki/Group_cohomology}


%
%
\end{thebibliography}
\end{document}